\documentclass{article}
\usepackage{amsmath,amssymb,amsthm}
\usepackage[svgnames, table]{xcolor} 
\usepackage{setspace}
\usepackage[colorlinks=true, linkcolor=blue, citecolor=blue, filecolor=blue, runcolor=blue, urlcolor = blue]{hyperref}

\usepackage{fullpage}
\newtheorem{definition}{Definition}[section]

\newtheorem{theorem}{Theorem}[section]

\newtheorem{proposition}{Proposition}[section]

\newtheorem{corollary}{Corollary}[section]

\newtheorem{remark}{Remark}[section]

\usepackage{graphicx,calc}
\DeclareGraphicsExtensions{.png}

\begin{document}

\title{The Jones polynomial of collections of open curves in $3-$space}

\author{Kasturi Barkataki $^1$ \and Eleni Panagiotou$^2$}

\date{%
    $^1$Department of Mathematics and SimCenter, University of Tennessee at Chattanooga, $613$ McCallie Ave, Chattanooga, $37403$, TN, USA.\\%
    $^2$Department of Mathematics and SimCenter, University of Tennessee at Chattanooga, $613$ McCallie Ave, Chattanooga, $37403$, TN, USA.\\[2ex]%
    \today
}

\maketitle

\begin{abstract}
Measuring the entanglement complexity of collections of open curves in 3-space has been an intractable, yet pressing mathematical problem, relevant to a plethora of physical systems, such as in polymers and biopolymers. In this manuscript, we give a novel definition of the Jones polynomial that generalizes the classic Jones polynomial to collections of open curves in 3-space. More precisely, first we provide a novel definition of the Jones polynomial of linkoids (open link diagrams) and show that this is a well-defined single variable polynomial that is a topological invariant, which, for link-type linkoids, it coincides with that of the corresponding link. Using the framework introduced in \cite{Panagiotou2020b}, this enables us to define the Jones polynomial of collections of open and closed curves in 3-space. For collections of open curves in 3-space, the Jones polynomial has real coefficients and it is a continuous function of the curves' coordinates. As the endpoints of the curves tend to coincide, the Jones polynomial tends to that of the resultant link. We demonstrate with numerical examples that the novel Jones polynomial enables us to characterize the topological/geometrical complexity of collections of open curves in 3-space for the first time.
\end{abstract}

\section{Introduction}

Many physical systems, such as polymers and biopolymers, textiles, and chemical compounds are composed by filamentous structures, that can be modeled by mathematical curves in 3-space, whose entanglement complexity determines their mechanical properties and function \cite{Arsuaga2005,Ricca2007,Taylor1974,Sulkowska2012,Qin2011,Panagiotou2019,Liu2018,Edwards1967}. Measuring multi-chain entanglement in such systems has remained an open problem for many decades \cite{Edwards1967,Edwards1968,DeGennes1974}. In this paper we introduce the first rigorous measure of complexity of collections of open curves in 3-space, via a traditional invariant of knots and links, the Jones polynomial. More precisely, the novel Jones polynomial that we introduce here generalizes the traditional Jones polynomial, so that it is applicable to collections of open curves in 3-space and gives a continuous measure of linking complexity which reduces to a topological invariant for closed curves.   

Collections of simple closed curves in 3-space (links) can be classified upon deformations without allowing cutting and pasting (topological equivalence). Topological invariants are functions defined on links that are invariant under Reidemeister moves and can be used to characterize the complexity of simple closed curves in 3-space. The notion of topological equivalence however is not useful for systems of open curves in 3-space, since any collection of open curves is topologically equivalent to any other. Instead of topological invariants, to characterize the topological complexity of open curves in 3-space, we seek measures of topological complexity that are continuous functions in the space of configurations. Until recently, the only measure of topological entanglement that could be applied to one or two open curves in 3-space to give a continuous measure of single or pairwise topological complexity, was the Gauss linking integral \cite{Gauss1877,Panagiotou2011, Panagiotou2013b,Panagiotou2014,Panagiotou2015}.  It was not until \cite{Panagiotou2020b}, where the necessary framework to define the Jones polynomial of a single open curve in 3-space was introduced, based on the notion of knotoids (open ended knot \textit{diagrams}) and their Jones polynomial. The Jones polynomial of an open curve is a polynomial with real coefficients that are continuous functions of the curve coordinates. This new framework also allowed to define Vassiliev measures of open curves in 3-space and to derive closed formulas for the second Vassiliev measure of single open curves in 3-space \cite{Panagiotou2021}. These advances led to immediate applications in materials and biology to obtain novel understanding of such physical systems, rigorously and without any closure scheme for the first time \cite{Herschberg2021,Smith2022,Wang2022}.   However, extending the Jones polynomial to collections of open curves in 3-space, has not been possible, even though one would think it would be straightforward, as the definition of the classical Jones polynomial applies to both knots and links. The reason for this is that an appropriate definition of the Jones polynomial of linkoids (open ended link \textit{diagrams}) is missing. This will be fully addressed in this manuscript.

The theory of knotoids was introduced by V.Turaev \cite{Turaev2012} as a means to study parts of knot diagrams with the aim to characterize knot complexity. Knotoids on a surface are open-ended knot diagrams and many analogous ideas of classical knot theory are applicable to the study of knotoids \cite{Gugumcu2017,Gugumcu2017b,Gugumcu2019, Manouras2021}. The Jones polynomial of knotoids follows from that of knots with a simple modification. The modification relies on assigning a value to states containing an open arc in the bracket expansion. The notion of knotoids can be naturally extended to multi-component cases, which we call linkoids (these can be seen as open link diagrams). Even though the definition of the Jones polynomial of knots extends naturally to that of links, the definition of the Jones polynomial of linkoids does not follow directly from that of knotoids and remains elusive. The difficulty consists on the fact that the states resulting from linkoids may contain several non-intersecting open arcs, which may connect end-points of different components, and it has been unclear how to deal with those in the polynomial. Very recently, in \cite{Neslihan2022} two definitions of the Jones polynomial were introduced to classify linkoids. Those definitions however, do not satisfy an important property, which prevents them to be used in giving an appropriate definition of the Jones polynomial of open curves in 3-space, as we will see later. In this manuscript, we will provide a novel definition of the bracket polynomial that can properly account for those states and which satisfies an important desired property, which enables the definition of a single variable Jones polynomial of linkoids and the Jones polynomial of collections of open curves in 3-space.



More precisely, in this paper we provide the first rigorous definition of the Jones polynomial of linkoids which has the property that the Jones polynomial of a link-type linkoid is equal to the Jones polynomial of the corresponding link. We use the framework introduced in \cite{Panagiotou2020b} to define the Jones polynomial of a collection of open curves in 3-space \cite{Panagiotou2020b}. Namely, any projection of a collection of open curves in 3-space to a plane corresponds to a linkoid diagram, whose Jones polynomial is now well defined. We define the Jones polynomial of the collection of open curves in 3-space as the average of the Jones polynomials in a projection over all possible projection directions. We prove that for a collection of open curves in 3-space, the newly introduced Jones polynomial is a polynomial with real coefficients which are continuous functions of the curve coordinates. We also prove that, as the endpoints of the curves move to coincide and form a link in 3-space, the Jones polynomial of the collection of open curves tends to that of the resulting link. We stress that the latter is possible only via the Jones polynomial of linkoids definition we introduce in this manuscript. The Jones polynomial introduced here provides a holistic definition of a Jones polynomial which applies to open and closed curves of single or multi-component systems. Our theoretical results are accompanied by an illustrative example of an open Borromean ring in 3-space. 

The contents of this paper are summarised as follows: In Section \ref{Linkoids}, we give the definitions for linkoid diagrams and linkoids and also make precise the notion of link-type linkoids. In Section \ref{JP_linkoids}, we provide all the necessary framework and then define the Jones polynomial of linkoids and we study its properties. In Section \ref{open curves}, we give the definition of the Jones polynomial of a collection of open curves in 3-space and we study its properties. Finally, in Section \ref{conclusion}, we present the conclusions of our results.

\section{Linkoids}
\label{Linkoids}

As mentioned in the previous section, the Jones polynomial of collections of open curves in 3-space will be defined via projections of the curves, which can be seen as linkoids. Even though we think of linkoids as projections of open curves, linkoids have been typically studied as diagrammatic objects. In this section, we present the definitions of linkoids and linkoid diagrams, originally defined in \cite{Gugumcu2017,Gugumcu2017b,Gugumcu2019, Manouras2021}, as well as the definition of link-type linkoids.

\begin{definition}\label{linkoid}(Linkoid Diagram)
Let $\Sigma$ be a surface. In this manuscript, we will consider  $\Sigma=S^2 = \mathbb{R}^2 \cup \infty$. A linkoid diagram $L$ with $n \in \mathbb{N}$ components in $\Sigma$ is a generic immersion of $[0,1] \times [0,1] \times \cdots \times [0,1]$ (upto $n$ times) in the interior of $\Sigma$ whose only singularities are transversal double points endowed with over/undercrossing data. These double points are called the crossings of $L$. The immersion of each $[0,1]$ is referred to as a component of the linkoid diagram and the images of $0$ and $1$ under this immersion are called the leg and the head of the component, respectively. These two points are distinct from each other and from the double points; they are called the endpoints of the component. The diagram $L$ has a total of $2n$ endpoints. A natural orientation is assigned for each $l_i$ from the leg to the head. 

For a linkoid diagram $L$ with $n$ components, we may introduce a convention to index all the legs by odd numbers $i \in \{1, 3, \cdots, 2n-1 \}$, and  the corresponding heads by even numbers $i+1 \in \{2, 4, \cdots, 2n \}$. Thus, a component of $L$ is denoted by $l_{\left(2j-1, 2j\right)} \quad $ such that $j \in \{1, 2 , \cdots, n \} \quad$ where the index $\quad 2j-1 \quad$ indicates the leg and the index $\quad 2j \quad$ indicates the head of the component. See Figure \ref{kdiags1} for examples of labeled linkoids.
\end{definition}

\begin{figure}[h]
    \centering
    \includegraphics[scale=0.12]{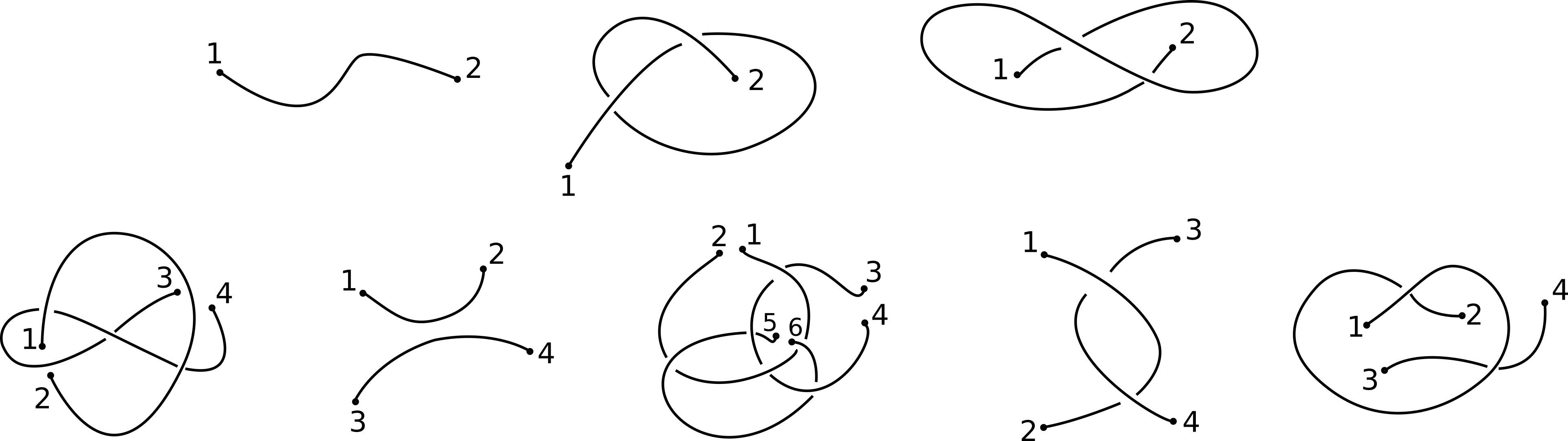}
    \caption{Examples of Knotoid and Linkoid Diagrams.}
    \label{kdiags1}
\end{figure}

\begin{definition}(Linkoid)
A linkoid is an equivalence class of the linkoid diagrams up to the equivalence relation induced by the three Reidemeister moves and isotopy. 
The three Reidemeister moves, denoted by $\Omega_1, \Omega_2, \Omega_3$, see Figure \ref{moves}(a), are defined on
linkoid diagrams and referred to as $\Omega$-moves. It is forbidden to pull the strand adjacent to an endpoint
over/under a transversal strand as shown in Figure \ref{moves}(b). These moves are called forbidden linkoid moves, and denoted by $\Phi_+$ and $\Phi_-$ , respectively.
\end{definition}

\begin{figure}[h]
    \centering
    \includegraphics[scale=0.3]{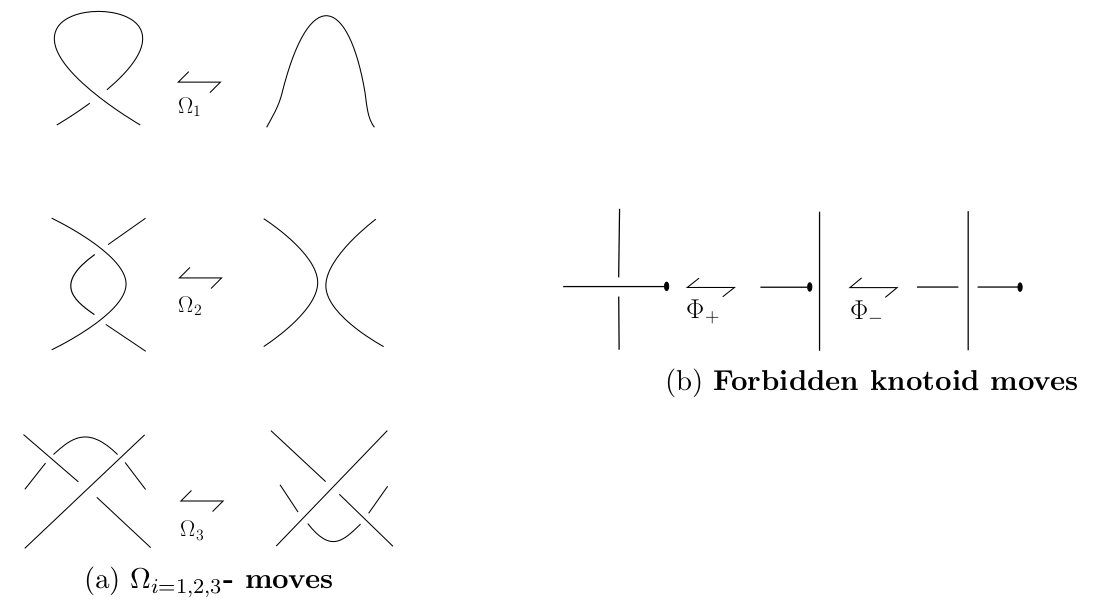}
    \caption{The $\Omega$-moves and Forbidden moves on a Linkoid Diagram. (See \cite{Gugumcu2017}). Note that these arcs are considered as parts of a larger diagram. }
    \label{moves}
\end{figure}

\begin{definition}\label{linktype}(Link-type Linkoid)
A linkoid is said to be of link-type if there exists a diagram in its equivalence class in which it is possible to draw a closure arc connecting the pair of endpoints (from the head to the leg) per component without introducing additional double points (crossings) to the diagram or between the closure arcs. A linkoid that is not of link-type, is called a proper linkoid.
\end{definition}

\begin{remark}
For knotoids (linkoids with 1 component), another way to distinguish between proper knotoids and knot-type knotoids is by checking whether there exists a diagram in which the two endpoints lie in the same region. This definition is consistent with Definition \ref{linktype}. Notice that, if for any diagram of the knotoid it is impossible to draw a closure arc without introducing new crossings, it follows that the endpoints must lie in different regions of the diagram. The converse is also true.
\end{remark}

\begin{remark}
Note that any pure braid (In a pure braid, the beginning and the end of each strand are in the same position) can be thought of as a link-type linkoid. Indeed, the closure arcs of a pure braid do not intersect any other arc, as do the closure arcs of a linkoid.

\end{remark}

\section{The Jones polynomial of Linkoids}
\label{JP_linkoids}

In this Section we will define the Jones polynomial of linkoids. As we mentioned in the Introduction, this will be later seen as the Jones polynomial of a projection of a collection of open curves in 3-space. However, it is also of interest in the study of linkoids in general. We will define the Jones polynomial of linkoids as the normalized Kauffman bracket polynomial of linkoids. 
Let us start with the classic smoothing skein relations, shown in Equation \ref{bkt_knotoids}. 

\begin{equation}
\langle\raisebox{-10pt}{\includegraphics[width=.05\linewidth]{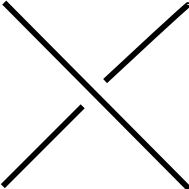}}\rangle=A\langle\raisebox{-10pt}{\includegraphics[width=.05\linewidth]{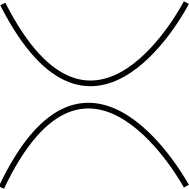}}\rangle+A^{-1}\langle\raisebox{-10pt}{\includegraphics[width=.05\linewidth]{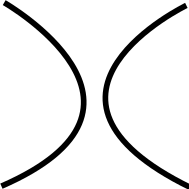}}\rangle,\hspace{0.5cm}\langle L\cup \bigcirc\rangle=(-A^2-A^{-2})\langle L\rangle.
\label{bkt_knotoids}
\end{equation}

Upon recursively smoothing a linkoid diagram using these relations, we obtain states with no crossings (see Figure \ref{borr1state}). 

\begin{definition}(State of a Linkoid Diagram)
A state $S$ of a linkoid diagram $L$ with $n$ components is an assignment of a choice of smoothing at each crossing. This results in a crossingless diagram with disjoint circles and $n$ long segments. Each long segment is labelled by the two endpoints that it connects. See Figure \ref{borr1state} for an illustrative example of a linkoid diagram and a corresponding state.
\label{def_st_linkoid}
\end{definition}

\begin{figure}[h]
    \centering
    \includegraphics[scale=0.15]{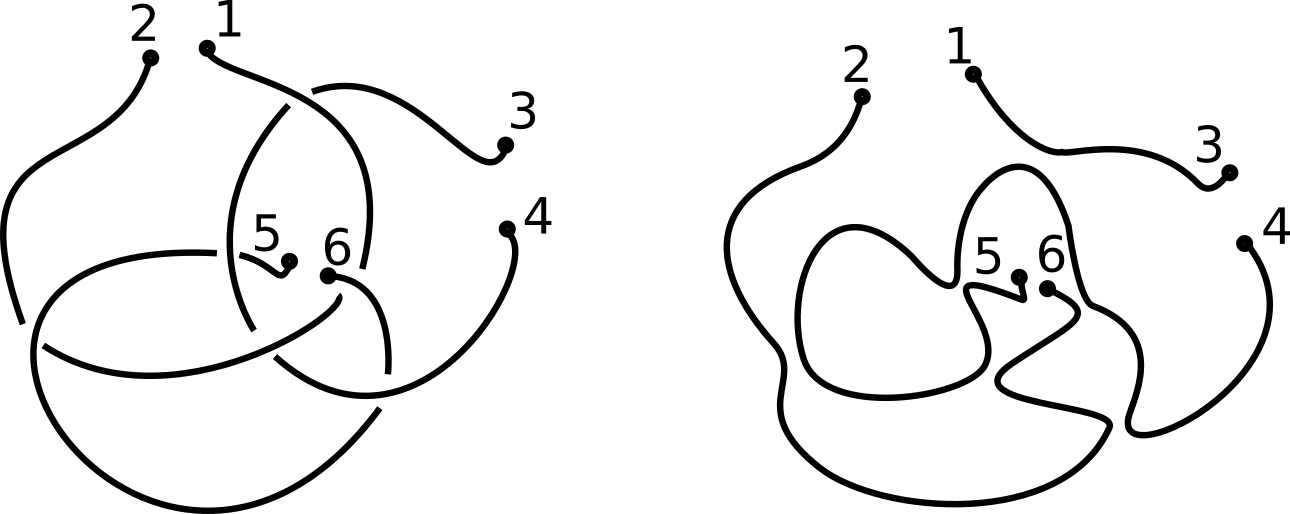}
    \caption{(Left) A linkoid diagram with 3 components, $l_{(1,2)}, l_{(3,4)}$ and $l_{(5,6)}$. (Right) One of the 64 possible states in the state sum expansion of the diagram. Clearly, this state is crossingless and contains 3 disjoint segments, namely $(1,3), (2,6)$ and $(4,5)$. Note that even though the number of long segments is the same as the number of components of the original linkoid, there has been a rearrangement in the pairing of endpoints per segment.}
    \label{borr1state}
\end{figure}

\begin{figure}[h]
    \centering
    \includegraphics[scale=0.15]{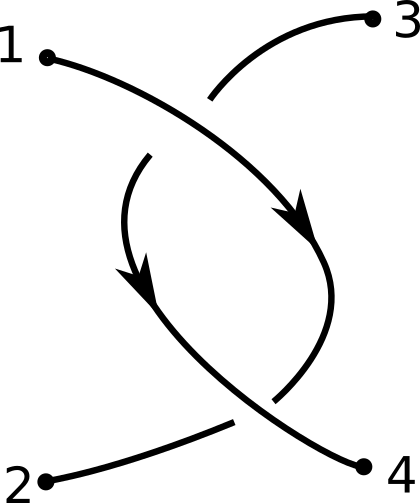}
    \caption{An oriented link-type linkoid with 2 components, corresponding to the Hopf link.  }
    \label{oriented_hopf_type}
\end{figure}

\noindent \textbf{Example 1 :}
To motivate the relations that will enable us to evaluate the bracket polynomial on these states, we will focus on a particular example, that of the linkoid diagram in Figure \ref{oriented_hopf_type}. Note that this linkoid diagram represents a link-type linkoid, corresponding to the Hopf link.  It is natural to require that the bracket (and consequently the Jones) polynomial of the linkoid in this example reflects the entanglement present in the Hopf link. Notice that the open Hopf-type linkoid, is equivalent to a Hopf link with two $\epsilon>0$ infinitesimal segments removed. Thus, it is natural to require that the bracket polynomial of the Hopf-type linkoid is in fact equal to that of the Hopf link, which is equal to $\left \langle \raisebox{-5pt}{\includegraphics[width=.05\linewidth]{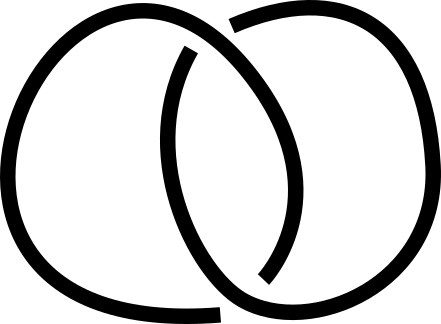}} \right \rangle =-A^4-A^{-4}$. This would in fact generalize the corresponding property of knot-type knotoids for the bracket (or Jones). Moreover, as we will see later, this not only a natural, but a necessary property to satisfy so that the Jones polynomial of collections of open curves in 3-space tends to that of the corresponding link, when the curves close, an essential desired property for measuring multi-curve complexity of curves in 3-space.

Consider the oriented linkoid consisting of 2 components as shown in Figure \ref{oriented_hopf_type}. By smoothing the crossings in the linkoid diagram using the bracket polynomial definition for knotoids (See Equation \ref{bkt_knotoids}), we get the following expansion :

\begin{equation}
\begin{split}
    \left \langle \raisebox{-12pt}{\includegraphics[width=.06\linewidth]{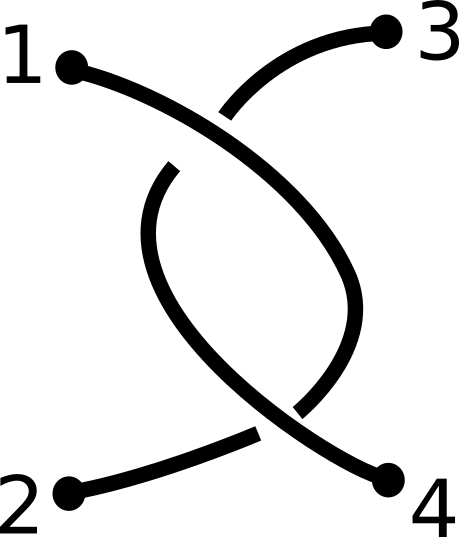}} \right \rangle &= A^2 \left \langle \raisebox{-12pt}{\includegraphics[width=.06\linewidth]{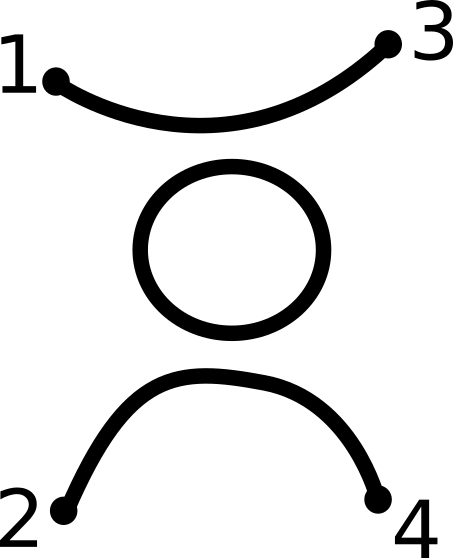}}\right \rangle + \left \langle \raisebox{-12pt}{\includegraphics[width=.06\linewidth]{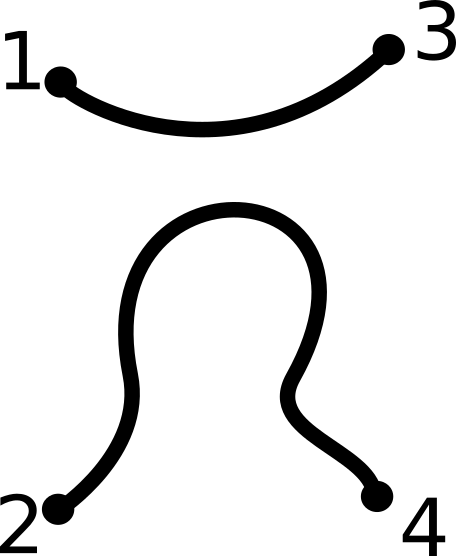}} \right \rangle + \left \langle \raisebox{-12pt}{\includegraphics[width=.06\linewidth]{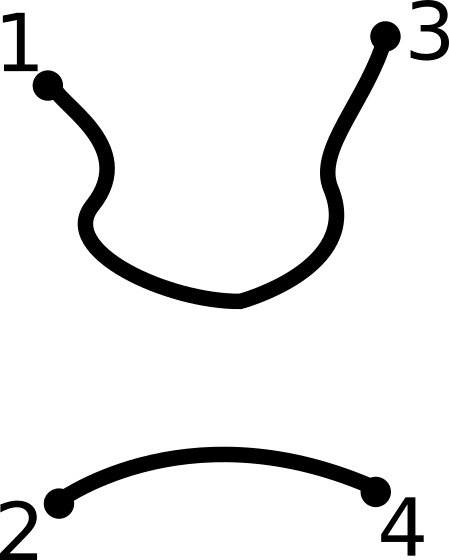}}\right \rangle + A^{-2} \left \langle \raisebox{-12pt}{\includegraphics[width=.06\linewidth]{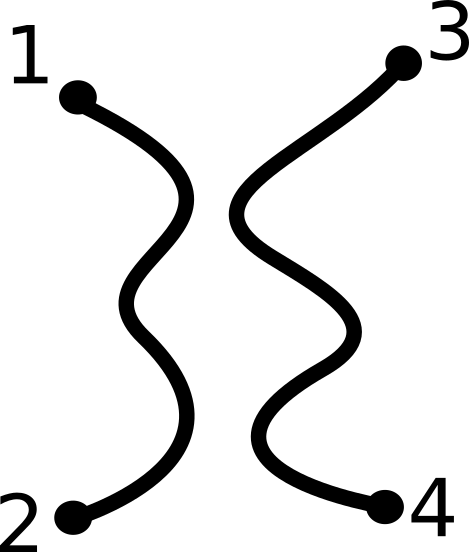}}\right \rangle\\
    &=A^2 d \left \langle \raisebox{-12pt}{\includegraphics[width=.06\linewidth]{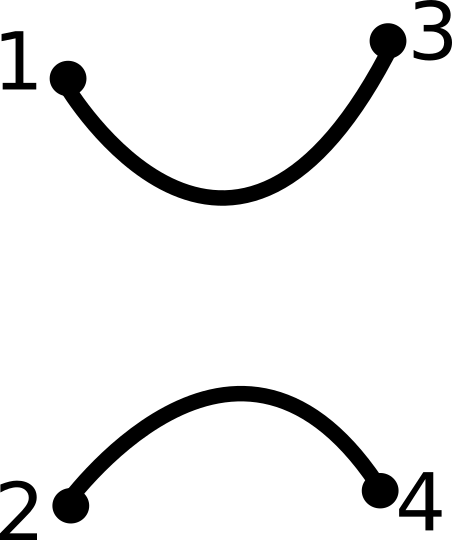}} \right \rangle + 2 \left \langle \raisebox{-12pt}{\includegraphics[width=.06\linewidth]{fig/h6.png}} \right \rangle + A^{-2} \left \langle \raisebox{-12pt}{\includegraphics[width=.06\linewidth]{fig/h5.png}} \right \rangle\\
    &=(A^2 d + 2) \left \langle \raisebox{-12pt}{\includegraphics[width=.06\linewidth]{fig/h6.png}} \right \rangle + A^{-2} \left \langle \raisebox{-12pt}{\includegraphics[width=.06\linewidth]{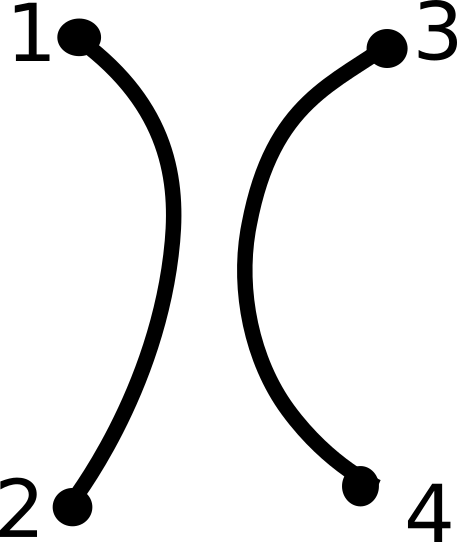}} \right \rangle\\
\end{split}
\label{hopf_type_bkt_1}
\end{equation}

\noindent where $d=-A^2-A^{-2}$.\\

The final expression in Equation \ref{hopf_type_bkt_1} is a summation of bracket polynomials of states of the linkoid diagram with two long segments each. Notice that the Kauffman bracket polynomial of knotoids can evaluate the bracket of an open arc, simply by setting \quad $\langle\includegraphics[width=.05\linewidth]{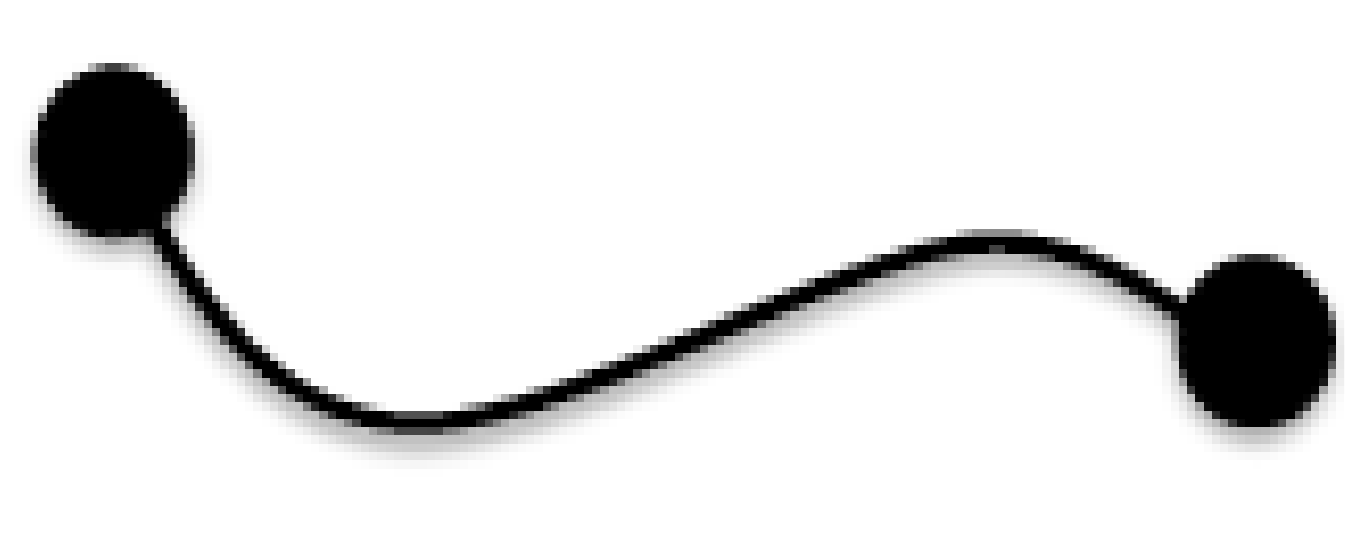}\rangle=1$, \quad but in linkoids we will have more than one long arc in a state, as shown in Equation \ref{hopf_type_bkt_1}. Moreover, even though the states will consist of two long segments, the segments connect different endpoints of the original components, as shown in Equation \ref{hopf_type_bkt_1}. For this reason, even though the Jones polynomial of links only accounts for the number of components in a state, this will not work for linkoids. For example, if we assign the same value to both of the cases in Equation \ref{hopf_type_bkt_1}, say $\left \langle\raisebox{-6pt}{\includegraphics[width=.04\linewidth]{fig/h6.png}} \right \rangle = \left \langle \raisebox{-6pt}{\includegraphics[width=.04\linewidth]{fig/h7.png}} \right \rangle = t, \quad$ then $\left \langle \raisebox{-6pt}{\includegraphics[width=.04\linewidth]{fig/h1.png}} \right \rangle = (A^2d+2+A^{-2})t = (-A^4+A^{-2}+1)t$, which is different than the bracket polynomial of the Hopf link, even if we set $t=d$. A very recent preprint (at the time of writing this manuscript)  provides a definition of the Jones polynomial of linkoids, where the authors propose to  simply introduce new variables, $\lambda_{i j}$ where $i<j$ refer to endpoints of the linkoid, for the different connecting arcs  \cite{Neslihan2022}.  Using that definition however, we get the expression of the bracket polynomial of the linkoid in Figure \ref{oriented_hopf_type} to be equal to $\quad ({-A}^{4}+1) \lambda_{1 3} \lambda_{2 4} + A^{-2} \lambda_{1 2} \lambda_{3 4}$, which is different from that of the Hopf link. Thus, none of the, up to date, proposed definitions satisfy the basic requirement that we set with this example.


In the following we will show that the definition we introduce in this manuscript will be a single variable, well-defined, topological invariant of linkoids, that satisfies the condition that the Jones polynomial of a link-type linkoid is that of the corresponding link (Theorem \ref{thm_linktype_jp}). To do this, we will first set the framework for analyzing states of linkoids in Section \ref{Comb}. This enables a novel definition of the Kauffman bracket and the Jones polynomials, which is given in Section \ref{Jones}. In Section \ref{linktype} we prove Theorem \ref{thm_linktype_jp}.

\subsection{Segment Cycles of a linkoid state}\label{Comb}

In this section we introduce an essential concept that will be necessary for classifying the states of a linkoid diagram, the segment cycles of a state. In the following, let $L$ be a linkoid diagram with $n \in \mathbb{N}$ components and let $G$ be the set containing all the endpoints (heads and legs) of $L$. Recall that, 
a component of a linkoid with $n$ components is denoted 
by $l_{(2j-1,2j)}$ where $j \in \{ 1, 2, \cdots, n\}$. Thus $G=\{ 1, 2, 3, \cdots , 2n\}$.

\begin{definition}(Pairing Combination and Head-Leg pairing) We will call an element $J \in S_{2n}$, which can be expressed as the product of $n$ disjoint $2-cycles$, a pairing combination. In particular, we will call the element $HL = (1\quad 2)(3\quad 4) \cdots (2n-1\quad 2n) = \prod_{i=1}^n (2i-1\quad 2i)\in S_{2n}$, the head-leg pairing. 

Note that any linkoid defines, by the connectivity of its components, a head-leg pairing (see Definition \ref{linkoid}), while a state, $S$, of a linkoid diagram, defines a pairing combination ,$J_S$ (which may, or may not, be a head-leg pairing). 



 


\label{def_pc}
\end{definition}



\begin{definition}(Orbit of an endpoint)
 Given an  endpoint $ a \in G$, and an arbitrary pairing combination, $J$ on $G$, the set $Orb_J(a)$ of $a$ is defined to be the orbit of the composition function $HL \circ J$. Symbolically, $Orb_J(a)$ is given as,

$$ Orb_J(a) = \left\{ x \in G \quad \mid \quad x = (HL\circ J)^{m}(a) \quad , \quad m \in \mathbb{Z}\right\}$$
\label{def_orbit}
\end{definition}

\begin{definition}(Segment Cycle)
Given a pairing combination, $J$ on $G$, the Segment Cycle of an endpoint $a \in G $ is  defined as follows:
$$Seg(a)=Orb_J(a) \sqcup Orb_J(HL(a))$$
where, $\quad Orb_J(a) \quad$ (resp. $Orb_J(HL(a))$) is the orbit of the point $a$ (resp. $HL(a)$) under the action of the pairing combination $J$. Note that for any point $a \in G$, $HL(a)\in G$ always belongs to the same segment cycle. Thus, a segment cycle always contains an even number of elements.

\label{def_seg_cyc}
\end{definition}

\begin{remark}\label{group}
The orbit of an endpoint can be also defined as the equivalence class of the endpoint (as an element of $G$) under the equivalence relation $\mathcal{R}_{\langle HL \circ J \rangle}$, induced by $\langle HL \circ J \rangle$, which is the action of the cyclic group generated by $HL\circ J \in S_{2n}$ on the set $G$. Let $\mathcal{G}=\langle HL \circ J \rangle \times \langle HL \rangle$. The segment cycle of an element $a \in G$ can also be defined as the orbit of $a$ under the action of $\mathcal{G}$ as shown below :
$$Sec(a)=Orb_{\mathcal{G}}(a) := \{ y \in G : \exists \quad (g,h) \in \mathcal{G} : y = (g, h) \star a\}$$ where, $\star$ means the group action such that $(g,h)\star a = g(h(a))$, the image of $a$ with respect to the composition function $g \circ h$. The quotient set $G/\mathcal{R}_\mathcal{G}$, where $\mathcal{R}_{\mathcal{G}}$ is the equivalence relation induced by $\mathcal{G}$, gives the set of all segment cycles. The cardinality of this set gives the total number of distinct segment cycles due to the pairing combination $J$. 

\end{remark}

\begin{proposition}
The number of segment cycles in a state $S$, $|S|_{cyc}$,  is bounded by: $\quad 1 \leq |S|_{cyc} \leq n.$.
\end{proposition}

\begin{proof}
If $J_S = HL$ (the trivial pairing combination), it follows from Remark \ref{group} that $\mathcal{G}=\langle HL \circ HL \rangle \times \langle HL \rangle = Id \times \langle HL \rangle$. Thus, $|G/\mathcal{R}_{\mathcal{G}}|=n$. 
If $J_S$ is single cycle, then $\mathcal{G}$ partitions the set $G$
into a single equivalence class. Therefore, $|G/\mathcal{R}_{\mathcal{G}}|=1$.

\end{proof}


\begin{remark}
\label{rem_decor_circ_J}
Let $S$ denote a state of $L$ with associated pairing $J_S$, and let $Seg(a)$ be a segment cycle in $L$, such that $|Seg(a)|=2k$. Then we can represent $Seg(a)$ by a circle decorated with the $2k$ endpoints of $L$ (See Figure \ref{decor1}), following the order in which they appear in the cycle (See Figure \ref{decor1}). Note that the arcs connecting 2 adjacent points in this circle  alternate between the functions, $J$ and $HL$. Any two points connected by $J$ belong to the same component of the state $S$ and any two points connected by $HL$ belong to the same component in $L$. 


\begin{figure}[h]
    \centering
    \includegraphics[scale=0.2]{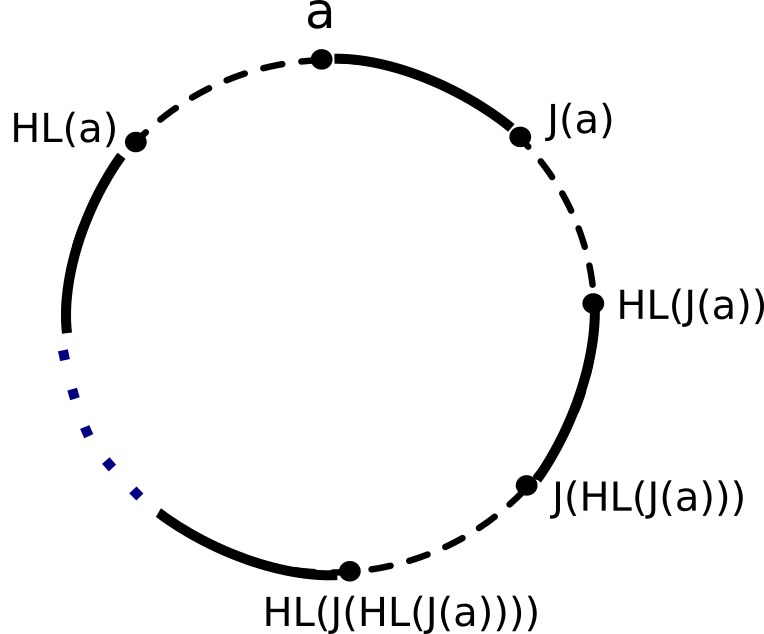}
    \caption{Representation of the segment cycle of $a \in G$ in terms of a decorated circle. Let us consider a circle and let $a \in G$ be the initial point on the circle. The remaining $2k-1$ endpoints  can be uniquely added into the circle in the order $\quad J(a), \quad HL(J(a)),\quad J(HL(J(a))),\quad \cdots ,\quad \text{upto } HL(a)$. Note that the arcs connecting 2 adjacent points in this circle  alternate between the functions, $J$ (solid) and $HL$ (dashed). Any two points connected by $J$ belong to the same component of the state $S$ and any two points connected by $HL$ belong to the same component in $L$. (The dotted line in the figure indicates the continued process of adding points.)}
    \label{decor1}
\end{figure}
\end{remark}




\vspace{20pt}

\noindent \textbf{ Example 1 (cont.) :}

Consider the linkoid diagram and the particular state (say $S$) as shown in Figure \ref{oriented_hopf_type}. In this example, the set $G$ of all endpoints is $\{ 1, 2, 3, 4\}$. Note that the open arc components of a linkoid diagram $L$ define a head-leg pairing, while the states of $L$ can define other pairing combinations. For example, let us denote by $s_1, s_2$ the final states of Equation \ref{hopf_type_bkt_1}. Then $s_1$ defines the pairing combination $J_{s_1}$ , which can be represented by the permutation $(1 \quad 3)(2 \quad 4) \in S_4$ and $s_2$ defines the pairing combination $J_{s_2}$ , which can be represented by the permutation $(1 \quad 2)(3 \quad 4) \in S_4$. The states $s_1$ and $s_2$ are explicitly shown in Figure \ref{js1}. 


For each of the permutations, $J_{s_1}$ and $J_{s_2}$, we can construct the corresponding set of segment cycles. Note that, a segment cycle can be represented as a decorated circle. Corresponding $J_{s_1}$ and $J_{s_2}$, the possible segment cycles are shown in Figure \ref{js1}.
\begin{figure}[h]
\centering
\includegraphics[scale=0.12]{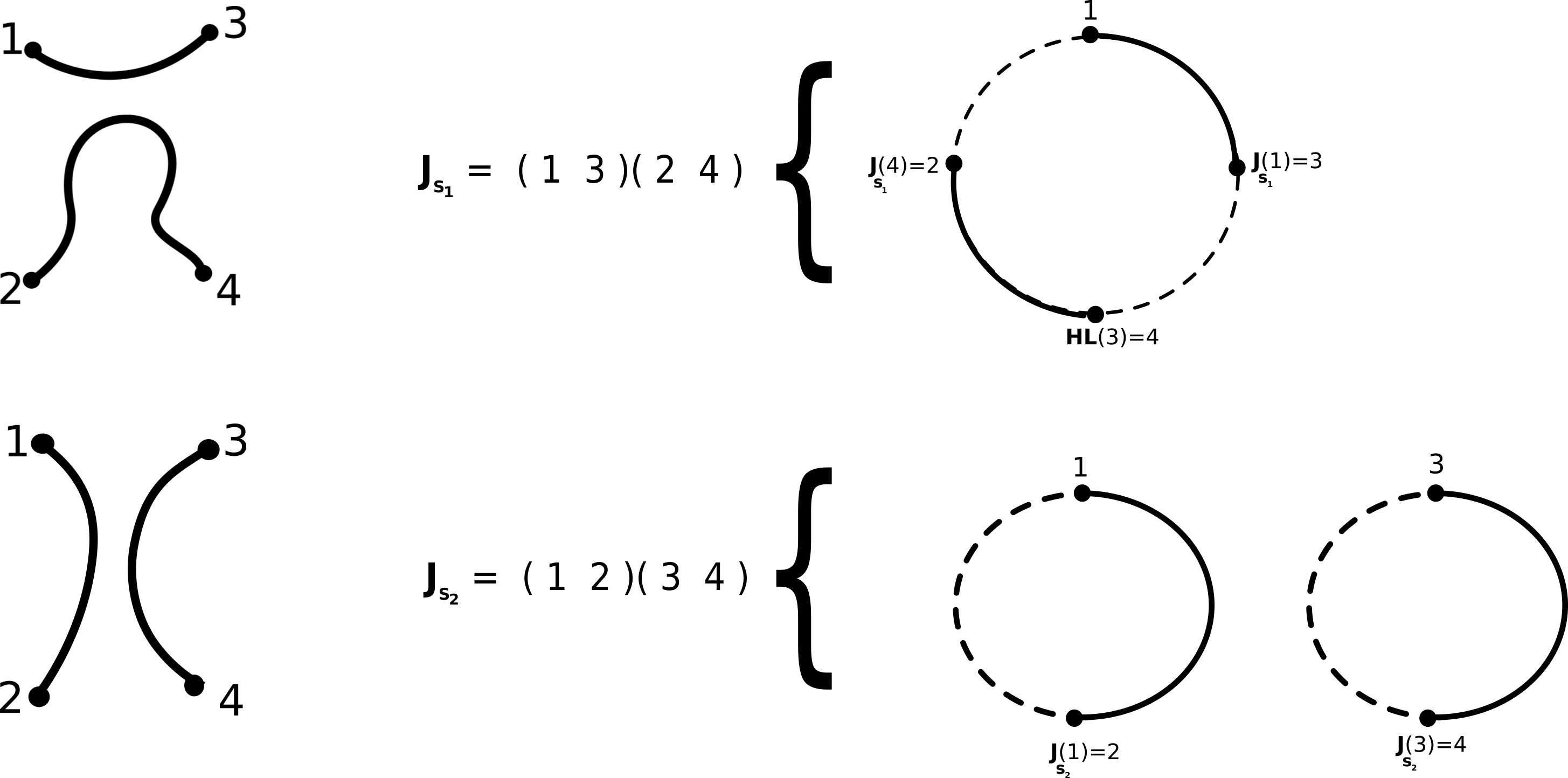}
\caption{(Top) The connections among endpoints due to state $s_1$, corresponding to the pairing combination $J_{s_1}=(1 \quad 3)(2 \quad 4) \in S_4$, and the resultant segment cycle. (Bottom) The connections among endpoints due to state $s_2$, corresponding to the pairing combination $J_{s_2}=(1 \quad 2)(3 \quad 4) \in S_4$, and the resultant distinct segment cycles.}
\label{js1}
\end{figure}

\subsection{The Jones polynomial of Linkoids}\label{Jones}

In the previous section we showed that open segments in a smoothed state of a linkoid can be grouped into segment cycles. We also saw the analogy between segment cycles and decorated circles. In this section, we use these concepts to define the bracket polynomial of linkoids and the Jones polynomial of linkoids (as the normalized bracket polynomial).  

\begin{definition}(Bracket Polynomial of a Linkoid)
Let $L$ be a linkoid diagram with $n$ components. The bracket polynomial of the linkoid is completely characterised by the following Skein relation and initial conditions:

\begin{equation}
\begin{split}
\left\langle\raisebox{-10pt}{\includegraphics[width=.06\linewidth]{fig/cross10.png}}\right\rangle=A\left\langle\raisebox{-10pt}{\includegraphics[width=.06\linewidth]{fig/cross30.png}}\right\rangle+A^{-1}\left\langle\raisebox{-10pt}{\includegraphics[width=.06\linewidth]{fig/cross20.png}}\right\rangle, & \hspace{0.5cm}\left\langle L\cup \bigcirc\right\rangle=(-A^2-A^{-2})\left\langle L\right\rangle,\\\hspace{0.5cm}\left\langle \raisebox{-12pt}{\includegraphics[width=.13\linewidth]{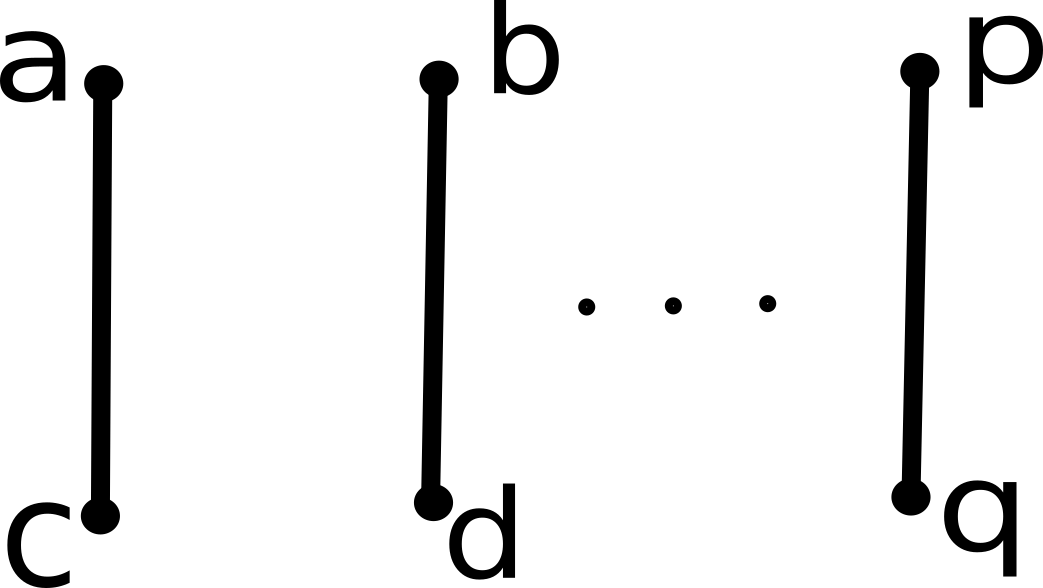}}\right\rangle=(-A^2-A^{-2})^{|cyc|}.& \\
\label{bkt_sk_L}
\end{split}
\end{equation}

\noindent where $|cyc|$ denotes the number of distinct segment cycles.
In other words, $|cyc|=|G/\mathcal{R}_\mathcal{G}|$, where $\mathcal{R}_{\mathcal{G}}$ is the equivalence relation induced on $G$ by $\mathcal{G}=\langle HL \circ J \rangle \times \langle HL \rangle$, where $J$ is the pairing defined by the indices $a, b, c, d, \cdots p, q$.\\
 \vspace{7pt}

\noindent The bracket polynomial of $L$ can be formulated as the following state sum expression :
\begin{equation}
    \langle L \rangle := \sum_S A^{\sigma (S)} d^{|S|_{circ} -1} d^{|S|_{cyc}}\\ 
\end{equation}
where, $S$ is a state corresponding to a choice of smoothing over all double points in $L$; $\sigma(S)$ is the algebraic sum of the smoothing labels of $S$; $|S|_{circ}$ is the number of disjoint circles in $S$ and $|S|_{cyc}$ is the number of distinct segment cycles in $S$. 
\label{def_bkt}
\end{definition}

\noindent The bracket polynomial of linkoids  has the following properties:

\begin{enumerate}
    \item It preserves the underlying Skein relation used in the computation of bracket polynomial of knots and knotoids. 
    \item If $L$ is a link diagram, the novel definition of the bracket polynomial (Definition \ref{def_bkt}) coincides with the traditional Kauffman bracket polynomial of $L$. Indeed, the states of $L$ have no long segments. Therefore $|S|_{cyc} = 0$ and $$\langle L \rangle = \sum_S A^{\sigma (S)} d^{|S|_{circ} -1}.$$
    \item If $L$ is a knotoid diagram, then the novel definition of the bracket polynomial (Definition \ref{def_bkt}) coincides with the Kauffman bracket polynomial of knotoids. Indeed, for a knotoid, $|S|_{cyc}=1$ for all states and
    $$\langle L \rangle = \sum_S A^{\sigma (S)} d^{|S|_{circ} -1}d^1 =  \sum_S A^{\sigma (S)} d^{|S|_{circ}}.$$
    \item For a trivial linkoid with  $n$ components, $$\langle L \rangle = d^{n-1}. $$ Indeed, for a trivial linkoid with  $n$ components, there are no rearrangements in the pairing of endpoints since there are no crossings to be resolved. Therefore there is only 1 state in the state sum which has the original $n$ segments and their endpoints intact and $$\langle L \rangle = A^0 \times d^{0-1} \times d^{n} = d^{n-1}. $$
    \item For link-type linkoids the bracket polynomial coincides with that of the corresponding link upon the closure of endpoints (see Theorem \ref{thm_linktype_jp}).
\end{enumerate}

In the following, the bracket polynomial turns into an invariant for linkoids with a normalization by the writhe giving rise to a definition for the Jones polynomial of linkoids with the substitution of $A=t^{-1/4}$. The writhe,  $Wr(L)$, of an oriented linkoid diagram $L$ is the algebraic sum of signs (positive or negative) of crossings of $L$.

\begin{definition}(Jones Polynomial of a Linkoid)
The normalized bracket polynomial of a linkoid diagram $L$ is defined as $$f_L = (-A^3) ^{-Wr(L)} \langle L \rangle$$
where  $Wr(L)$ is the writhe of the linkoid diagram and $\langle L \rangle$ is as seen in Definition \ref{def_bkt}. This gives the Jones polynomial of a linkoid with the substitution $A=t^{-1/4}$.
\label{def_jones}
\end{definition}

 
\noindent The Jones polynomial of linkoids (See Definition \ref{def_jones}) has the following properties:

\begin{enumerate}
\item The Jones polynomial of a linkoid (See Definition \ref{def_jones}) is a topological invariant of linkoids and it satisfies the Jones polynomial Skein relations :
$$t^{-1}f_{L_+}-t f_{L_-}=(t^{1/2}-t^{-1/2})f_{L_0}.$$
where, the linkoids $L_+, \quad L_-$ and $L_0$ are identical almost everywhere except at one crossing as shown below :
\begin{center}
    \includegraphics[scale=0.12]{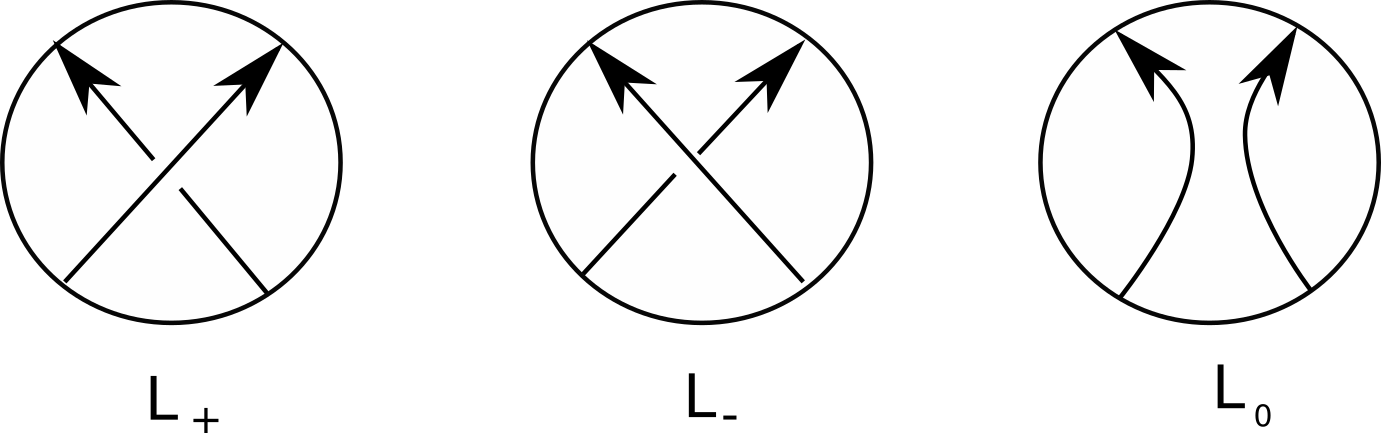}
\end{center}
        \item If $L$ is a link diagram, then the new definition of the Jones polynomial (Definition \ref{def_jones}) gives the traditional Jones polynomial of the link. Indeed, for a link $L$, there is no long segment. Therefore $|S|_{cyc} = 0$ and $$f_L  = (-A^3)^{-Wr(L)}\sum_S A^{\sigma (S)} d^{|S|_{circ} -1} d^0.$$ 
    \item If $L$ is a knotoid diagram, the Jones polynomial of $L$ defined here coincides with the Jones polynomial of knotoids. Indeed, if $L$ has only one component, then $|S|_{cyc}=1$ (only 1 long segment) for all states and $$f_L  = (-A^3)^{-Wr(L)}\sum_S A^{\sigma (S)} d^{|S|_{circ}}.$$
    \item For a trivial linkoid with  $n$ components, $$f_L  = d^{n-1}. $$ Indeed, in that case there are no rearrangements in the pairing of endpoints since there are no crossings to be resolved. Therefore there is only 1 state in the state sum which has the original $n$ segments and their endpoints intact.  Since there are no crossings, the writhe is zero and this allows us to express the Jones polynomial as
    $$f_L  =(-A^3)^0 A^0 \times d^{0-1} \times d^{n} = d^{n-1}. $$
    \item For link-type linkoids the Jones polynomial coincides with that of the corresponding link upon the closure of endpoints (see Theorem \ref{thm_linktype_jp}).
\end{enumerate}
\bigskip

\noindent \textbf{ Example 1 (cont.) :}

Let us return to the example of a linkoid diagram with 2 components as given in Figure \ref{oriented_hopf_type}. Here, the set $G$ of endpoints is equal to $\{ 1, 2, 3, 4\}$. The final step of the state sum expansion of its bracket polynomial (See Equation \ref{hopf_type_bkt_1}) involves the bracket polynomials of the two states $ \quad \raisebox{-8pt}{\includegraphics[width=.05\linewidth]{fig/h6.png}}$ and $\quad \raisebox{-8pt}{\includegraphics[width=.05\linewidth]{fig/h7.png}}$. Let us call them $s_1$ and $s_2$ respectively and their respective pairing combinations as $J_{s_1}$ and $J_{s_2}$. Note that $J_{s_2} = HL = (1 \quad 2)(3 \quad 4)$ implies that there are 2 distinct segment cycles corresponding to $s_2$, namely $Seg(1)=\{1, 2\}$ and $Seg(3)=\{3,4\}$. Whereas, the action of $J_{s_1} = (1 \quad 3)(2 \quad 4)$ on $G$ gives $Seg(1)=\{1,2,3,4\}=G$. Therefore, there is only 1 distinct segment cycle corresponding to $s_1$.

Using Definition \ref{def_bkt}, the final step of Equation \ref{hopf_type_bkt_1} can be simplified as follows :
\begin{equation}
\begin{split}
    \left \langle \raisebox{-12pt}{\includegraphics[width=.06\linewidth]{fig/h1.png}} \right \rangle 
    &=(A^2 d + 2) \left \langle \raisebox{-12pt}{\includegraphics[width=.06\linewidth]{fig/h6.png}} \right \rangle + A^{-2} \left \langle \raisebox{-12pt}{\includegraphics[width=.06\linewidth]{fig/h7.png}} \right \rangle\\
    &=(A^2d+2)d^{0-1}d^1 + A^{-2}d^{0-1}d^{2}\\
    &=-A^{4}-A^{-4}\\
\end{split}
\label{hopf_type_bkt_2}
\end{equation}

Notice that the above expression matches the bracket polynomial for the Hopf link, i.e. \\$\left \langle \raisebox{-5pt}{\includegraphics[width=.04\linewidth]{fig/hc.png}} \right \rangle =-A^4-A^{-4}$. The writhe of the diagram in Figure \ref{oriented_hopf_type} is $Wr(\raisebox{-3pt}{\includegraphics[width=.02\linewidth]{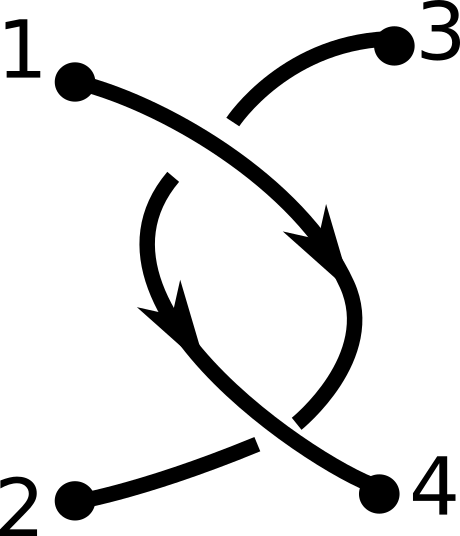}}) = -2$. Therefore, the Jones polynomial of the linkoid is evaluated as follows :
\begin{equation}
    \begin{split}
         f_{\raisebox{-6pt}{\includegraphics[width=.03\linewidth]{fig/h0_dark.png}}} &= (-A^3)^{-Wr(\raisebox{-3pt}{\includegraphics[width=.02\linewidth]{fig/h0_dark.png}})} \left \langle \raisebox{-12pt}{\includegraphics[width=.06\linewidth]{fig/h1.png}} \right \rangle \\
        &=(-A^3)^2 (-A^4-A^{-4})\\
        &=-A^{10}-A^{2}\\
    \end{split}
\end{equation}

\noindent which is equal to that of the Hopf link. In fact, we will prove that this is the case for all link-type linkoids.

\begin{figure}[h]
\centering
(a) \quad \includegraphics[scale=0.15]{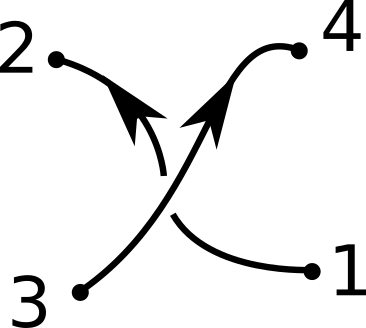}  \quad \quad \quad (b) \includegraphics[scale=0.15]{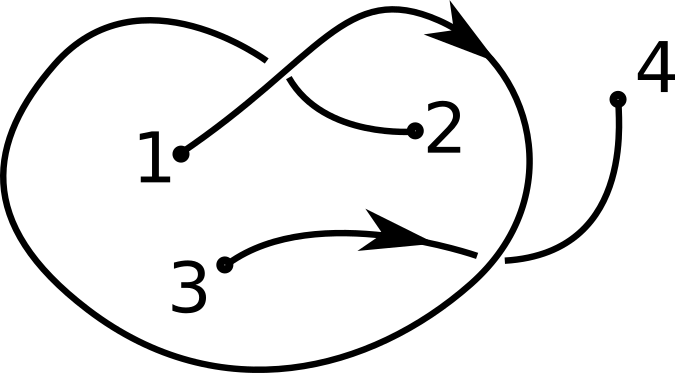}
\caption{Pure linkoids i.e. non-link-type linkoids. Clearly Fig (a) has Writhe = +1 and Fig (b) has Writhe = 0.}
\label{properK}
\end{figure}

The Jones polynomial (Definition \ref{def_jones}) distinguishes non-equivalent linkoids in $S^2$ and for proper linkoids, this measure is not that of any link. For example the Jones polynomial of the linkoid given in Figure \ref{properK} (a) is equal to \quad  $f_{\raisebox{-6pt}{\includegraphics[width=.04\linewidth]{fig/properK2.png}}}=A + A^{-1}$ \quad whereas, the Jones polynomial of the linkoid given in Figure \ref{properK} (b)  is equal to \quad  $f_{\raisebox{-6pt}{\includegraphics[width=.04\linewidth]{fig/properK1.png}}}=-A^{-2} + A^{-4}$.

\subsection{The Jones polynomial of Link-type Linkoids}\label{linktype}

In this section we will prove that the Jones polynomial of link-type linkoids is equal to the Jones polynomial of the corresponding link. 

\begin{theorem}
\label{thm_linktype_jp}
Let $L$ be a link-type linkoid with $n$ components and $L_c$ be the corresponding link (the link that results from connecting the head to the leg of each component in a way that no new crossing is created). Then the Jones polynomials of $L$ and $L_c$ are equal, that is $f_L = f_{L_c}$.
\end{theorem}

\begin{proof}
By the definition of $L$ and $L_c$, we notice that $L$ can be created from $L_c$ by omitting $n$ arcs (closure arcs). These arcs connect the endpoints $2j-1$ and $2j$ for each component $l_{(2j-1,2j)}$ of $L$, where $j \in G=\{1, 2, \cdots, n\}$. Let us construct a decoration on $L_c$ by $L$, by keeping track of the endpoints of $L$ with labels. Thus $L_c$ is a union of $L$ with closure arcs. 


The closure arcs in $L$ do not create any new crossings, hence the total number of double points in $L_c$ is the same as that in $L$. Therefore, $Wr(L) = Wr(L_c)$. In the following, we show that the bracket polynomials, $\langle L \rangle$ and $\langle L_c \rangle$ are also equal.

We know that the states of a link or a linkoid diagram are completely determined by the choice of smoothing at the crossings of the diagram.  Since $L$ and $L_c$ have identical crossings, for each state $S$ of the linkoid diagram $L$, there exists a state $S_c$ in the link diagram $L_c$, such that their smoothing labels are equal, that is, $\quad \sigma(S) = \sigma(S_c)$. 
By definition, the contribution of $S$ in the state sum expression of $\quad \langle L \rangle \quad$ is $ \quad A^{\sigma(S)} d^{|S|_{circ}  + |S|_{cyc}-1} \quad$, where $|S|_{circ}$ is the number of disjoint circles and 
$|S|_{cyc}$ is the number of segment cycles formed by the disjoint long segments in the state $S$.  
Similarly, the contribution of $S_c$ in the state sum expression of $\quad \langle L_c \rangle \quad$ is $\quad A^{\sigma(S)} d^{|S_c|_{circ} -1}. \quad$ We will prove that these two terms are equal by showing that $\quad |S|_{circ}+|S|_{cyc}= |S_c|_{circ}$. 


Since the closure arcs of $L$ in $L_c$ do not introduce crossings, the state $S$ can be obtained from  $S_c$, by omitting the closure arcs. Thus, for a decorated $L_c$, $S_c$ will also be decorated with an even number of labels on each circle (since any closure arc has two endpoints and closure arcs are disjoint). Notice that if a circle in $S_c$ is not decorated, then it corresponds to a circle in $S$.
Let us express the total number of circles in $S_c$ as  $|S_c|_{circ} = |S_c|_{(circ,u)} + |S_c|_{(circ,d)} \quad ,$ where $\quad |S_c|_{(circ,u)}$ and $|S_c|_{(circ,d)} \quad$ denote the number of undecorated and decorated circles, respectively. Thus, $\quad |S|_{circ} = |S_c|_{(circ,u)}$. Therefore, we need to prove that $|S|_{cyc}= |S_c|_{(circ,d)}$.



Once all the circles that do not involve endpoints are taken care of in both $S$ and $S_c$, we are left with $n$ long segments in $S$ and $|S_c|_{(circ,d)}$ number of decorated circles in $S_c$. 
Since $S$ is formed by $S_c$ by removing the closure arcs, a decorated circle in $S_c$ with $2k$ labels gives rise to $k$ long segments in $S$. 
We will show that these $k$ long segments form a segment cycle. 

Notice that for a decorated circle in $S_c$, two adjacent labels are either endpoints of a closure arc, or a new connection formed by the smoothings. So, each consecutive pair of labels in a decorated circle in $S_c$ that does not belong to a closure arc, it defines a pairing $J_S$ in $S$ and each consecutive pair of labels in the decorated circle in $S_c$ that belong to the same closure arc, they are related by $HL$. Then the corresponding $k$ long segments in $S$ define a segment cycle. 

These $k$ long segments cannot overlap with any other decorated circle of $S_c$ because it will violate the fact that all the decorated circles of $S_c$ are disjoint. Therefore, for every decorated circle in $S_c$ there is a unique collection of long segments in $S$ such that these long segments form a segment cycle. This implies, $|S_c|_{(circ,d)} = |S|_{cyc}$. Therefore, $\langle L \rangle = \langle L_c \rangle$.  
\end{proof}

\section{The Jones polynomial of Open Curves in 3-space}
\label{open curves}

Consider a collection of $m \in \mathbb{N}$ open or closed curves in 3-space in general position. A (regular) projection of these collection of $n$ curves can give rise to a different linkoid diagram based on the choice of the direction of projection. Notice that with probability one, a projection will be generic. We use the framework introduced in \cite{Panagiotou2020b} and Definition \ref{def_bkt}, to rigorously define the bracket and Jones polynomials of a collection of $n\in \mathbb{N}$ open curves in 3-space. We will define the Jones polynomial as the normalized bracket polynomial.

\begin{definition}\label{jones3}
Let $\mathcal{L}$ denote a collection of $m \in \mathbb{N}$ open curves in 3-space. Let $\mathcal{L}_{\vec{\xi}}$ denote the projection of $\mathcal{L}$ on a plane with normal vector $\vec{\xi}$. The normalised bracket polynomial of $\mathcal{L}$ is defined as:
\begin{equation}
f_{\mathcal{L}}=\frac{1}{4 \pi} \int_{\vec{\xi}\in S^2} (-A^3)^{-Wr(\mathcal{L}_{\vec{\xi}})}\langle \mathcal{L}_{\vec{\xi}} \rangle dS.
    \label{jones_open_curves}
\end{equation}
where, each $\mathcal{L}_{\vec{\xi}}$ is a linkoid diagram and its bracket polynomial can be calculated by using Definition \ref{def_bkt}. Note that the integral is taken over all vectors $\vec{\xi} \in S^2$ except a set of measure zero (corresponding to the irregular projections). This gives the Jones polynomial of a collection of open curves in 3-space with the substitution $A=t^{-1/4}$.
\end{definition}

This new definition of the Jones polynomial of collections of open or closed curves in 3-space generalizes all the previous definitions of the Jones polynomial, so that it satisfies the following properties:

\begin{enumerate}
    \item The Jones polynomial defined by Equation \ref{jones_open_curves} does not depend on any particular projection of the collection of open or closed curves.
    \item For a collection of open curves this polynomial is not the polynomial of a corresponding/approximating
link, nor that of a corresponding/approximating linkoid.
    \item The Jones polynomial of a collection of open curves in $3-$space has real coefficients. It is not a topological invariant, but it is a continuous function of the curve coordinates (see Proposition \ref{cont}).
    \item For a collection of closed curves in $3-$space (a link), the Jones polynomial defined in Equation \ref{jones_open_curves} gives the traditional Jones polynomial and it can be computed from a single projection, i.e. \quad $f_{\mathcal{L}}=f_{\mathcal{L}_{\vec{\xi}}}$ \quad where, $\vec{\xi} \in S^2$ is any projection vector.
    \item As the endpoints of a collection of open curves in 3-space tend to coincide, the Jones polynomial tends to that of the corresponding link.
    \item For a linkoid of 1 component (a knotoid), the Jones polynomial of Definition \ref{jones3} gives the Jones polynomial defined in \cite{Panagiotou2020b}.
\end{enumerate}

In the case of polygonal curves in 3-space, the Jones polynomial attains a simpler expression. Without loss of generality, suppose that all the curves have $n$ edges each. Then there exists a finite number (Say $k \in \mathbb{N}$) of distinct linkoid types ($L_i$)  which may occur in any projection of the collection of open curves, $\mathcal{L}$. Therefore, Equation \ref{bkt_open_curves} can also be expressed as the following finite sum:

\begin{equation}
f_{\mathcal{L}} = \sum_{i=1}^k p_i f_{L_i}
    \label{finte_sum_bkt}
\end{equation}

\noindent where $p_i$ denotes the geometric probability that a projection of $\mathcal{L}$ gives the linkoid $L_i$.



\begin{proposition}\label{cont}
Let $\mathcal{L}$ denote a collection of simple open curves in 3-space. Then $f_{\mathcal{L}}$ is a continuous function of the co-ordinates of $\mathcal{L}$.

\begin{proof}
Let us approximate $\mathcal{L}$ by a set of polygonal curves of $n$ edges each, we denote $\mathcal{L}^{(n)}$. 
Then 

\begin{equation}
f_{\mathcal{L}^{(n)}} = \sum_{i=1}^k p_i f_{L^{(n)}_i}
    \label{finte_sum_bkt}
\end{equation}

\noindent where $L^{(n)}_i, i=1, \dotsc,k$ are the possible linkoids that can occur in all projections of $\mathcal{L}^{(n)}$ and $p_i$ the corresponding geometric probabilities. 
The geometric probability $p_i$ can be expressed as 
$$p_i = \frac{2A_0}{4\pi}$$
where $A_0$ is the Area on the sphere corresponding to vectors $\vec{\xi} \in S^2$ such that the projection of $\mathcal{L}^{(n)}$ along such vectors results in the linkoid $L_i^{(n)}$. $A_0$ is a quadrangle bounded by great circles defined by the edges and vertices of the polygonal curves in $\mathcal{L}^{(n)}$. Thus it is a continuous function of the coordinates of $L^{(n)}$ (see proof of Lemma 3.1 in \cite{Panagiotou2020b}).
The result follows as $n$ goes to infinity.
\end{proof}
\end{proposition}

\begin{corollary}
Let $L$ denote a collection of open curves in 3-space. As the endpoints of the curves tend to coincide to form a link $L_c$, $f_{\mathcal{L}}$ tends to $f_{\mathcal{L}_c}$
\end{corollary}

\begin{proof}
The result follows by Proposition \ref{cont}, Definition \ref{jones3} and Theorem \ref{thm_linktype_jp}.
\end{proof}

The statement below follows as a corollary of the properties of open and closed curves in 3-space that we have established so far.

\begin{corollary}
The Jones polynomial is a continuous function in the space of all simple curves (open or closed) in 3-space.
\end{corollary}

In a similar way, we can define the Kauffman bracket polynomial of a collection of open curves in 3-space, as follows:

\begin{definition}\label{bracket3d}
Let $\mathcal{L}$ denote a collection of $n \in \mathbb{N}$ open curves in 3-space. Let $\mathcal{L}_{\vec{\xi}}$ denote the projection of $\mathcal{L}$ on a plane with normal vector $\vec{\xi}$. The bracket polynomial of $\mathcal{L}$ is defined as:
 \begin{equation}
 \langle \mathcal{L} \rangle = \frac{1}{4\pi}\int_{\vec{\xi}\in S^2} \langle \mathcal{L}_{\vec{\xi}} \rangle dS
    \label{bkt_open_curves}
\end{equation}
where, each $\mathcal{L}_{\vec{\xi}}$ is a linkoid diagram and its bracket polynomial can be calculated by using Definition \ref{def_bkt}. Note that the integral is taken over all vectors $\vec{\xi} \in S^2$ except a set of measure zero (corresponding to the irregular projections).
\end{definition}

\noindent Properties of the bracket polynomial of collections of open curves in 3-space :

\begin{enumerate}
    \item The bracket polynomial defined in Equation \ref{bkt_open_curves} does not depend on any particular projection of the collection of curves.
    \item For an open curve this polynomial is not the polynomial of a corresponding/approximating
link, nor that of a corresponding/approximating linkoid.
    \item The bracket polynomial defined in Equation \ref{bkt_open_curves} is not a topological invariant, but it is a continuous function of the curve coordinates for both open and closed curves in 3-space.
    \item As the endpoints of a collection of open curves in 3-space tend to coincide, the bracket polynomial tends to that of the corresponding link.
    \item For a linkoid of one component, the bracket polynomial of Definition \ref{bracket3d} gives the bracket polynomial defined in \cite{Panagiotou2020b}.

\end{enumerate}

\noindent \textbf{ Example 2 :}

Consider a set of open borromean rings realised in 3-space by the following three lists of co-ordinates : 

\begin{verbatim}
   R = [[0,0,0],[1,1,0],[2,2,0.5],[3,3,0.5],[4,4,0],[5,5,0],[6,6,0.5],
   [7,7,0.5],[8,7,0.5], [9,5,0.2],[9,3,0.2],[8,0,0.2],[8,-1,0.2],
   [6,-1.5,0],[4,-2,0],[2,-1.5,0]]
   
   B = [[1,0,0.5],[4,0,0],[5,1,0],[5,4,0.5],[4,5,0.5],[3,6,0],
   [2,7,0],[-1,6,0],[-1,3,0.5]]
    
   K = [[6,0,0.5],[7,6,0],[6,7,0],[3,7,0.5],[2,6,0.5],[2,3,0],
   [3,2,0],[4,1,0.5]]    
\end{verbatim}

\noindent where {\tt R, B} and {\tt K} denote the red, blue and black curves, respectively. The list \quad {\tt R} \quad can be updated by an additional element (say $\vec{r}$) by using the following parametrization :

$$\vec{r}=\vec{r_1}+ s (\vec{r_2}-\vec{r_1})$$

\noindent where, \quad $0\leq s \leq 1$ \quad and $\vec{r_1}$ and $\vec{r_2}$ are  respectively the the last and the first points in \quad {\tt R} \quad i.e. $\vec{r_1}=\text{{\tt R}}[-1]$ and $\vec{r_2}=\text{{\tt R}}[0]$. Using the same parameter $s$, the lists {\tt B} and {\tt K} can each be updated by an additional element. 

Let us denote a configuration of the system of \textit{open borromean rings} by $w(s)$, where $s$ is the value of the parameter in concern. Clearly, the initial configuration of the system of \textit{open borromean rings} can thus be denoted as $w(0)$. As we start updating the lists {\tt R, B} and {\tt K} by the above parametrization, it means that the endpoints of each component of $w(0)$ move closer and closer in time, eventually attaining the configuration of the closed borromean rings, namely $w(1)$. During this journey from $w(0)$ to $w(1)$, we encounter infinitely many intermediate configurations, some of which are : $w(0.22)$, $w(0.44)$, $w(0.67)$, $w(0.68)$, $w(0.70)$, and $w(0.89)$. The Jones polyomials of the collection of curves from the initial to the final stage, along with the aforementioned intermediate steps are presented in Figure \ref{borr_in_time} and their explicit expressions are given as follows:

\begin{figure}
    \centering
    \includegraphics[scale=0.18]{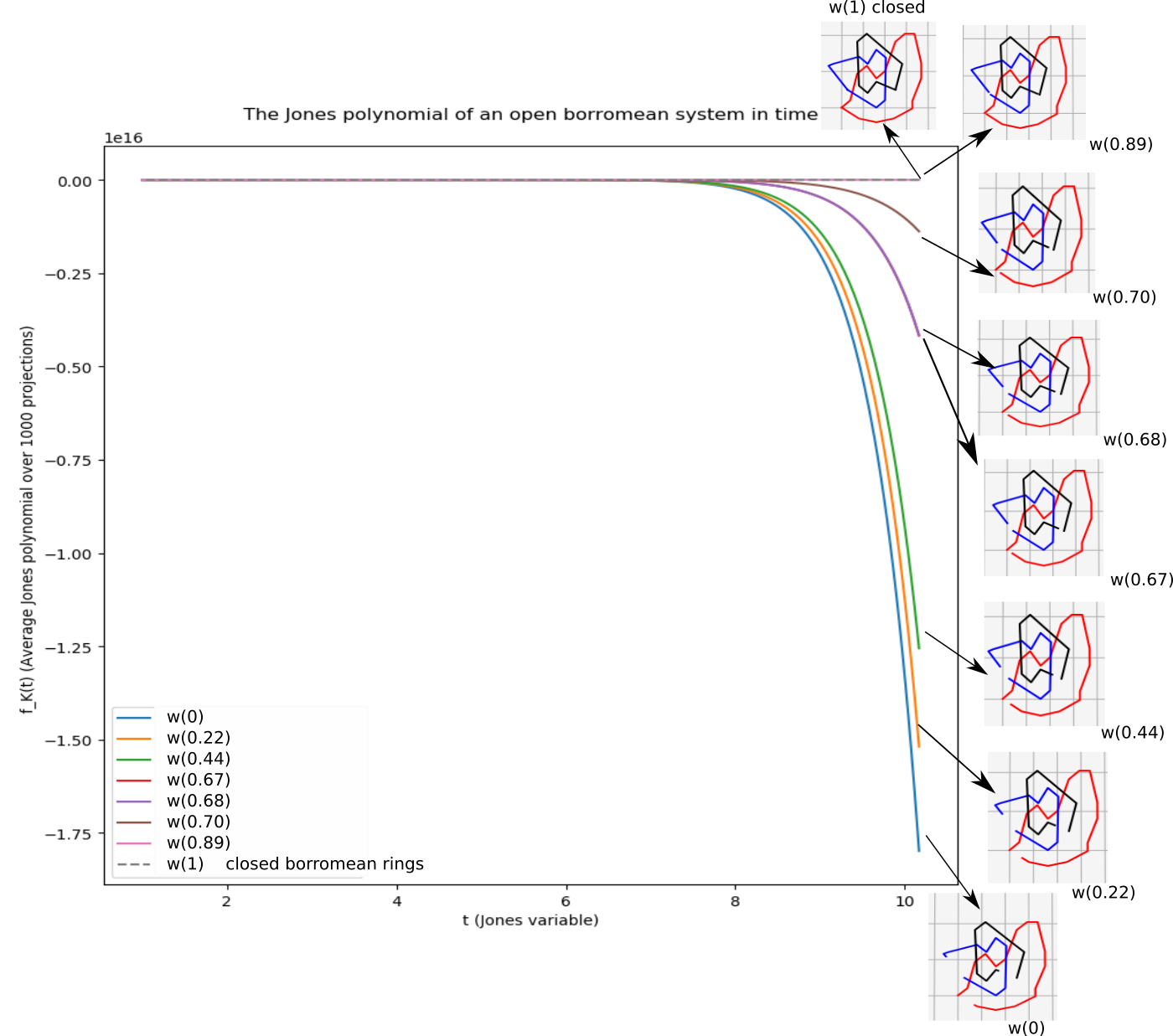}
    \caption{The Jones polynomial of a system of open borromean rings in 3-space as the end points move closer in time to ultimately give rise to the closed borromean rings in 3-space. The coefficients of the Jones polynomial are a continuous function of the chain coordinates.}
    \label{borr_in_time}
\end{figure}

\newpage 

\begin{tabular}{c c c}
     $f_{w(0)}$ &= & $ -0.26 t^{-3} +1.49 t^{-2} + 1.84 t^{-3/2} + 0.16 t^{-1} -0.38 t^{-1/2}$\\
     & & + $0.72 + 0.67 t^{1/2} -0.18 t -0.22 t^{3/2} +0.22 t^2 -0.07 t^3$\\
     \\
     $f_{w(0.22)}$ & =&  $ -0.59 t^{-3} +2.15 t^{-2} + 1.83 t^{-3/2} -0.81 t^{-1} -1.02 t^{-1/2}$\\
     & & + $1.39 + 1.59 t^{1/2} -0.21 t -0.53 t^{3/2} -0.27 t^2 -0.08 t^3$\\
     \\
     $f_{w(0.44)}$ & =&  $ -0.92 t^{-3} +2.83 t^{-2} + 1.67 t^{-3/2} -1.86 t^{-1} -1.53 t^{-1/2}$\\
     & & + $2.27 + 2.39 t^{1/2} -0.45 t -0.85 t^{3/2} +0.53 t^2 -0.16 t^3$\\
     \\
     $f_{w(0.67)}$ &= & $ -0.98 t^{-3} +2.95 t^{-2} + 1.45 t^{-3/2} -2.04 t^{-1} -1.4 t^{-1/2}$\\
     & & + $2.60 + 2.22 t^{1/2} -0.68 t -0.78 t^{3/2} +0.86 t^2 -0.27 t^3$\\
     \\
     $f_{w(0.68)}$ & =& $ -0.98 t^{-3} +2.94 t^{-2} + 1.43 t^{-3/2} -2.02 t^{-1} -1.39 t^{-1/2}$\\
     & & + $2.60 + 2.19 t^{1/2} -0.68 t -0.77 t^{3/2} +0.88 t^2 -0.28 t^3$\\
     \\
     $f_{w(0.70)}$ & =& $ -0.98 t^{-3} +2.96 t^{-2} + 1.38 t^{-3/2} -2.04 t^{-1} -1.34 t^{-1/2}$\\
     & & + $2.66 + 2.11 t^{1/2} -0.74 t -0.75 t^{3/2} +0.96 t^2 -0.31 t^3$\\
     \\
     $f_{w(0.89)}$ &= & $ -0.99 t^{-3} +2.98 t^{-2} + 0.18 t^{-3/2} -2.06 t^{-1} -0.17 t^{-1/2}$\\
     & & + $3.86 + 0.35 t^{1/2} -1.9 t -0.15 t^{3/2} +2.74 t^2 -0.9 t^3$\\
     \\
     $f_{w(1)}$ &= & $ -t^{-3} + 3 t^{-2} -2 t^{-1} + 4 -2 t  + 3 t^2 -t^3$\\
     \\
\end{tabular}

Notice that $f_{w(0)}$ is a new polynomial representing the particular configuration of the open Borromean ring in 3-space. This is a polynomial with real coefficients, while $f_{w(1)}$ is the integer polynomial invariant of the Borromean ring. We notice that, as the endpoints of the open link tend to coincide, the coefficients of the powers of $t$ that compose the Borromean ring tend to their corresponding integer values, while the coefficients of the powers of $t$ that are not part of the Borromean ring, tend to zero.

\section{Conclusions}
\label{conclusion}

In this work we introduced the first measure of topological complexity of collections of open curves in 3-space, based on a novel Jones polynomial. The classical Jones polynomial is a special case of this novel Jones polynomial. For collections of open curves in 3-space, the novel Jones polynomial is a polynomial with real coefficients, which are continuous functions of the curve coordinates and as the endpoints of the curves tend to coincide, it tends to the integer coefficient, Jones polynomial invariant of the resulting link.

The definition of the Jones polynomial of open curves in 3-space is based on a novel definition of the Jones polynomial of linkoids that we introduced in this manuscript as well. This novel Jones polynomial of linkoids is the only such definition that satisfies the basic property that the polynomial of a link-type linkoid is that of corresponding link. This polynomial thus generalizes the Jones polynomial of knotoids, while maintaining its properties. This new definition of the Jones polynomial of linkoids, will enable to properly define other invariants of linkoids as well in the future.

We demonstrated with numerical examples how the novel Jones polynomial of open curves in 3-space can be useful in practice to characterize multi-chain complexity for the first time. This enables the rigorous characterization of multi-chain entanglement in many physical systems obtained either from experiments or simulations, such as polymers and biopolymers, where entanglement is arguably an important factor of mechanics and function, which has been elusive.


\section{Acknowledgements}
 Kasturi Barkataki and Eleni Panagiotou were supported by NSF (Grant No. DMS-1913180 and NSF CAREER 2047587).

\bibliographystyle{plain}
\bibliography{paperDatabase}

\end{document}